\documentclass[a4paper,12pt]{article}
\usepackage[utf8]{inputenc}
\usepackage{amsmath}
\usepackage{amssymb}
\usepackage{amsthm}
\usepackage{mathtools}
\usepackage{cite}
\usepackage{booktabs}

\usepackage{multirow}
\usepackage{hyperref}
\usepackage{bbm}

\title{Numerical Treatment of Non-local Integral Operators in the Framework of Evolutionary Equations}
\author{Sebastian Franz\footnote{
          Institute of Scientific Computing, TU Dresden, Germany.
          \mbox{e-mail}: sebastian.franz@tu-dresden.de}
        \and
        Sascha Trostorff\footnote{
           Mathematisches Seminar, CAU Kiel, Germany.
           \mbox{e-mail}: trostorff@math.uni-kiel.de}
       }
\date{\today}

\usepackage{fancyhdr} 
\usepackage{xparse}

\newcommand{\scp}[2][H]{\left\langle #2 \right\rangle_{#1}}

\newcommand{\scprm}[1]{\langle #1 \rangle_{\rho,m}}

\newcommand{\N}{\mathbb{N}}
\newcommand{\R}{\mathbb{R}}
\newcommand{\PS}{\mathbb{P}}
\newcommand{\U}[2][q]{\PS_{#1}^{\textrm{disc}}(#2)}

\newcommand{\tmi}{{t_{m,i}}}
\newcommand{\Qm}[2][m]{Q_{#1}\left[#2\right]}
\newcommand{\Qmr}[2][\rho]{Q_m\left[#2\right]_{#1}}
\DeclareMathOperator{\Div}{{div}}
\DeclareMathOperator{\Curl}{{curl}}
\newcommand{\Grad}{\nabla}
\newcommand{\lin}{\mathrm{lin}\,}
\DeclareMathOperator*{\esssup}{{ess-sup}}
\newcommand{\dom}{\operatorname{dom}}

\newcommand{\e}{\mathrm{e}}
\newcommand{\id}[1]{\mathbbm{1}_{#1}}
\newcommand{\pt}{\partial}
\newcommand{\dt}{\,\mathrm{d}t}
\newcommand{\ds}{\,\mathrm{d}s}
\newcommand{\norm}[2]{\|{#1}\|_{#2}}
\newcommand{\lrarrow}{\quad\Leftrightarrow\quad}
\newcommand{\rarrow}{\quad\Rightarrow\quad}
\newcommand{\jump}[1]{[\hspace*{-2pt}[#1]\hspace*{-2pt}]}
\newcommand{\pmtrx}[1]{\ensuremath{\begin{pmatrix}#1 \end{pmatrix}}}
\newcommand{\spmtrx}[1]{\ensuremath{\left(\begin{smallmatrix}#1 \end{smallmatrix}\right)}}
\newcommand{\vecsymb}[1]{\boldsymbol{#1}}
\newcommand{\vx}{\vecsymb{x}}
\newcommand{\vq}{\vecsymb{q}}
\newcommand{\ord}[1]{\mathcal{O}\left(#1\right)}

\theoremstyle{plain}
\newtheorem{thm}{Theorem}[section]
\newtheorem{lem}[thm]{Lemma}
\newtheorem{prop}[thm]{Proposition}

\theoremstyle{definition}
\newtheorem{defi}[thm]{Definition}
\newtheorem{rem}[thm]{Remark}
\newtheorem{prob}[thm]{Problem}
\begin{document}
  \pagestyle{fancy}
  \maketitle
  \begin{abstract}
    Using the theory of evolutionary equations, we consider abstract differential equations including non-local integral operators.  After providing a condition
    for the well-posedness of the addressed equation we consider a numerical method of approximating its solution. We provide convergence
    proofs under conditions on the kernel of the integral operator and the solution and finish the paper with
    some simulation results.
  \end{abstract}


     
  \section{Introduction}
  
  In 2009 Rainer Picard published his seminal paper \cite{Picard} where an abstract Hilbert space framework for dealing with a huge class of linear differential equations (which arise in mathematical physics) was presented. In its simplest form, the equations under consideration are given by
  \begin{equation}\label{eq:evo_eq}
    \overline{(\pt_t M_0+M_1+A)}U=F,
  \end{equation}
  where $A$ is (usually) an unbounded operator on a Hilbert space $H$, incorporating the differential operators with respect to the spatial variables, and $M_0$ and $M_1$ are bounded operators on $H$ incorporating physical quantities of the underlying medium (e.g. denisty, heat conductivity, permeability etc.). All operators, including the temporal derivative $\pt_t$ are established in a weighted $L^2$-space, given by
  \[
    L^{2,\rho}(\R;H):=
      \left\{ 
          f:\R\to H\,:\, f \mbox{ meas.}, 
          \|f\|_{2,\rho}^2\coloneqq \int_\R \|f(t)\|_H^2 \exp(-2\rho t) \dt<\infty
      \right\},
  \]
  where $\rho\in \R$. The main result of \cite{Picard} is the well-posedness of \eqref{eq:evo_eq} under mild conditions on the operators involved:
  \begin{thm}\label{thm:sol_theory}
   Let $A:\dom(A)\subseteq H\to H$ be skew-selfadjoint, $M_0,M_1$ bounded operators on $H$ such that $M_0$ is selfadjoint. Moreover, assume that there exists $\rho_0\in \R$ and $c>0$ such that
   \[
    \rho \langle M_0 x,x\rangle_H+\frac{1}{2} \langle (M_1+M_1^\ast) x,x\rangle_H \geq c\|x\|^2_H \quad (\rho\geq \rho_0, x\in H).
   \]
    Then for each $\rho\geq \rho_0$ the problem \eqref{eq:evo_eq} is well-posed in the sense, that for each $F\in L^{2,\rho}(\R;H)$ there exists a unique $U\in L^{2,\rho}(\R;H)$ satisfying \eqref{eq:evo_eq} and
    \[
     \|U\|_{2,\rho}\leq \frac{1}{c} \|F\|_{2,\rho}.
    \]
   Moreover, one has causality; that is, if $F=0$ on $]-\infty,a]$ for some $a\in \R$, then also $U=0$ on $]-\infty,a]$.
  \end{thm}

  
  The basic theory has later been extended to many more fields of application,
   see e.g.~\cite{McGPTW20, STW22, Trostorff2013, PTWW2013, Trostorff2020}.\\
   In this paper we consider an extension
  including an integral term: 
  For a given $F\in L^{2,\rho}(\R;H)$ find $U\in L^{2,\rho}(\R;H)$, such that
  \begin{equation}\label{eq:problem}
    \overline{(\pt_t M_0+M_1+A+T_K)}U=F,
  \end{equation}
  where
  \[
    T_KU(t)=\int_{-\infty}^t K(t,s)U(s)\ds
  \]
  is an integral operator with a given kernel $K$. In \cite{Trostorff2015} the setting of $T_K$ being a convolution
  operator was considered (see also \cite{Trostorff_Habil} and \cite{PTW2015} for the treatment of fractional derivatives
  and integrals). Here we allow $K$ to be more general and analyse in Section~\ref{sec:wellposedness} the well-posedness of the problem.
  
  Problems involving integral operators with kernels include Volterra and Fredholm integro-differential equations,
  convolution kernels, and weakly singular kernels, such as those used in fractional derivatives of Caputo or Riemann--Liouville type.
  There are many papers in the literature on how to numerically approximate solutions to these problems. We will mention some of them in the following. 
  Time-stepping methods such as Runge--Kutta or multistep methods can be used for Volterra integro-differential equations, see e.g \cite{ZhangVandewalle2006},
  or in the form of convolution quadrature using the Laplace transform of the kernel for convolution kernels, see e.g. \cite{Lubich2004,Schaedle2006}.
  An alternative approach uses collocation methods with continuous or discontinuous 
  piecewise polynomial functions, see, e.g., the monograph \cite{Brunner2004} or for spectral collocation methods the paper \cite{Sheng2014}.
  Fractional integro-differential equations have their own numerical schemes, the most prominent of which is the L1 scheme, 
  see, e.g., \cite{KopMC19} or the standard reference \cite{Diet10}.
  
  In this paper, we will use a different numerical method.
  We will approximate problem \eqref{eq:problem} on a finite time interval and consider as domain of interest
  $[0,T]\times\Omega$, where $T>0$ is a fixed end-time and $\Omega\subset\R^d$ is a given domain such that $H=L^2(\Omega)^n$ for some $n\in \N$.
  The numerical method will then be a Galerkin finite element method using conforming discrete spaces in the spatial variable 
  and a discontinuous Galerkin setting for the time variable. This is again a time-stepping method, but of variational type.

  We will extend
  our convergence theory from \cite{FrTW16} and provide the setting of the method in Section~\ref{sec:method}
  before we analyse its convergence behaviour in Section~\ref{sec:numanalysis}. We close the paper with some examples
  in Section~\ref{sec:examples}.

  \section{Well-posedness}\label{sec:wellposedness}
  In this section we address the well-posedness of \eqref{eq:problem}. For doing so, we first introduce the class of admissible kernels $K$. For this section, we denote
  \[
    \Lambda := \{(t,s) \in \R \times \R \,;\, s \le t\}.
  \]

  \begin{defi}
    For $\rho \in \R$ we define the space
    \begin{align*}
      L^{1,\rho}_{\mathrm{unif}}(\Lambda) :=
      \bigg\{ K : \Lambda \to \R \text{ measurable} \,;\,
      &\esssup_{t \in \R} \int_{-\infty}^{t} |K(t,s)| \e^{-\rho (t-s)} \ds<\infty,
      \\
      &\esssup_{s \in \R} 
      \int_{s}^{\infty} |K(t,s)| \e^{-\rho (t-s)} \dt < \infty
      \bigg\},      
    \end{align*}
    where we identify—as usual—functions which agree almost everywhere. Moreover, we set
    \[
      \norm{K}{1,\rho,\mathrm{unif}} :=
      \max \left\{\esssup_{t \in \R} \int_{-\infty}^{t} |K(t,s)| \e^{-\rho (t-s)} \ds,\;
                  \esssup_{s \in \R} \int_{s}^{\infty} |K(t,s)| \e^{-\rho (t-s)} \dt\right\},
    \]
    for any $K \in L^{1,\rho}_{\mathrm{unif}}(\Lambda)$.
  \end{defi}

  \begin{rem}$\,$
  \begin{enumerate}
    \item[a)] It is straightforward to check that $\left(L^{1,\rho}_{\mathrm{unif}}(\Lambda), \norm{\cdot}{1,\rho,\mathrm{unif}}\right)$ is a Banach space.\\[4pt]
    \item[b)] In \cite{Trostorff2015} the space
    \[
     L^{1,\rho}(\R_{\geq 0})\coloneqq \{\ell : \R_{\geq 0}\to \R \text{ measurable}\,;\, \int_0^\infty |\ell(t)|\e^{-\rho t} \dt <\infty\}
    \]
    was considered for kernels of convolution operators. For $\ell : \R_{\ge 0} \to \R$ measurable, we set
              \[
               K_\ell : \Lambda \to \R, \quad K_\ell(t,s) := \ell(t-s).
              \]
              Then $K_\ell \in L^{1,\rho}_{\mathrm{unif}}(\Lambda)$ if and only if $\ell \in L^{1,\rho}(\R_{\ge 0})$ and
              \[
                \norm{K_\ell}{1,\rho,\mathrm{unif}} = \norm{\ell}{1,\rho}.
              \]
  \end{enumerate}
  \end{rem}

  \begin{lem}
    Let $\rho, \mu \in \R$ with $\rho \le \mu$. Then
    \[
      L^{1,\rho}_{\mathrm{unif}}(\Lambda) \subseteq L^{1,\mu}_{\mathrm{unif}}(\Lambda)
      \quad \text{with} \quad
      \norm{K}{1,\mu,\mathrm{unif}} \le \norm{K}{1,\rho,\mathrm{unif}}
    \]
    for $K \in L^{1,\rho}_{\mathrm{unif}}(\Lambda)$.
  \end{lem}

  \begin{proof}
    Let $K \in L^{1,\rho}_{\mathrm{unif}}(\Lambda)$. Then for $t \in \R$
    \[
      \int_{-\infty}^{t} |K(t,s)| \e^{-\mu (t-s)} \ds \le
      \int_{-\infty}^{t} |K(t,s)| \e^{-\rho (t-s)} \ds,
    \]
    as well as
    \[
      \int_{s}^{\infty} |K(t,s)| \e^{-\mu (t-s)} \dt \le
      \int_{s}^{\infty} |K(t,s)| \e^{-\rho (t-s)} \dt
    \]
    for $s \in \R$. These inequalities yield the assertion.
  \end{proof}

  \begin{prop}\label{prop:Opbound}
    Let $\rho \ge 0$ and $K \in L^{1,\rho}_{\mathrm{unif}}(\Lambda)$. We set
    \[
      T_K : L^{2,\rho}(\R) \to L^{2,\rho}(\R), \qquad
      (T_K f)(t) := \int_{-\infty}^{t} K(t,s) f(s)\, \ds.
    \]
    Then $T_K$ is well-defined and bounded with
    \[
      \norm{T_K}{} \le \norm{K}{1,\rho,\mathrm{unif}}.
    \]
    Moreover, $T_K$ is causal in the sense that, if $f$ vanishes on some interval $]-\infty,a]$ for some $a\in \R$, then so does $T_Kf$.
  \end{prop}

  \begin{proof}
    Let $f \in C_c(\R)$ and set $a := \inf \operatorname{spt} f$. Then we estimate for $t \ge a$
    \[
      \int_{-\infty}^{t} |K(t,s)|\,|f(s)|\,\ds
        \leq \norm{f}{\infty} \int_a^t |K(t,s)|\,\ds
        \leq \norm{f}{\infty} \int_a^t |K(t,s)| \e^{-\rho (t-s)} \e^{\rho (t-a)} \,\ds < \infty,
    \]
    which shows that $\int_{-\infty}^{t} K(t,s)f(s)\, \ds$ exists for each $t \in \R$ (for $t < a$ the integral is zero).

    Now we estimate for $t \in \R$
    \begin{align*}
      \bigg| \int_{-\infty}^{t}&K(t,s)f(s)\, \ds \, \e^{-\rho t} \bigg|\\
        &= \left| \int_{-\infty}^{t} K(t,s) \e^{-\rho (t-s)} f(s)\e^{-\rho s} \ds \right| \\
        &\leq \int_{-\infty}^{t} |K(t,s)|^{1/2} \e^{-\rho (t-s)/2}
              |K(t,s)|^{1/2}  \e^{-\rho (t-s)/2} |f(s)| e^{-\rho s} \ds \\
        &\leq \left(\int_{-\infty}^{t} |K(t,s)| \e^{-\rho (t-s)} \ds\right)^{1/2}
              \left(\int_{-\infty}^{t} |K(t,s)| \e^{-\rho (t-s)} |f(s)|^2 \e^{-2\rho s} \ds\right)^{1/2}.
    \end{align*}
    Taking squares and integrating with respect to $t$ yields
    \begin{align*}
      \int_{\R} \left| \int_{-\infty}^{t} K(t,s)f(s)\, \ds \right|^2 \e^{-2\rho t} \dt
        &\leq \|K\|_{1,\rho,\mathrm{unif}} \int_{\R} \int_s^{\infty} |K(t,s)| \e^{-\rho (t-s)} \dt |f(s)|^2 \e^{-2\rho s} \ds\\
        &\leq \|K\|_{1,\rho,\mathrm{unif}}^2 \|f\|_{2,\rho}^2,
    \end{align*}
    which yields the asserted continuity estimate. The causality of $T_K$ is obvious.
  \end{proof}

  Now let $H$ be a Hilbert space, $M_0,M_1 \in L(H)$ with $M_0$ selfadjoint and 
  $A:\operatorname{dom}(A)\subseteq H \to H$ a skew-selfadjoint operator,
  such that there exist $\rho_0 \geq 0$, $c>0$ such that for all $\rho \geq \rho_0$
  \[
    \rho M_0 + \frac{1}{2}(M_1+M_1^\ast) \geq c,
  \]
  as in Theorem \ref{thm:sol_theory}.
  \begin{thm}
    Let $\rho_0 \geq 0$ and $K \in L^{1,\rho_0}_{\mathrm{unif}}(\Lambda)$ with
    $\norm{K}{1,\rho_0,\mathrm{unif}} < c$. Then for each $\rho \geq \rho_0$ and
    $F \in L^{2,\rho}(\R;H)$ the equation
    \[
      \overline{(\partial_{t} M_0 + M_1 + A + T_K)}U = F
    \]
    admits a unique solution $U \in L^{2,\rho}(\R;H)$. Moreover, the solution operator
    \[
      S_\rho := \overline{(\partial_{t} M_0 + M_1 + A + T_K)}^{-1}
    \]
    is bounded and causal.
  \end{thm}

  \begin{proof}
    From Picard’s Theorem (see Theorem \ref{thm:sol_theory}) we know that
    \[
      C_\rho := \overline{(\partial_{t} M_0 + M_1 + A)}^{-1}
    \]
    is bounded and causal with $\norm{C_\rho}{} \leq \frac{1}{c}$.
    Now we observe that
    \begin{align*}
      \overline{(\partial_{t} M_0 + M_1 + A + T_K)}U = F
      &\lrarrow
      \overline{(\partial_{t} M_0 + M_1 + A)}U = F - T_K U\\
      &\lrarrow
      U = C_\rho(F - T_K U),      
    \end{align*}
    since $T_K$ is a bounded operator. Thus, we are seeking fixed points of the mapping $Q : U \mapsto C_\rho(F - T_K U)$.
    Since $\norm{T_K}{} \leq \norm{K}{1,\rho_0,\mathrm{unif}} < c$, uniqueness and existence follow
    by the contraction mapping principle. 

    Moreover, continuity of $S_\rho$ follows by
    \[
      \norm{U}{2,\rho} 
        \leq \norm{C_\rho}{}(\norm{F}{2,\rho} + \norm{T_K U}{2,\rho})
        \leq \frac{1}{c} \norm{F}{2,\rho} + \frac{\norm{K}{1,\rho_0,\mathrm{unif}}}{c}\norm{U}{2,\rho},
    \]
    and hence,
    \[
      \norm{U}{2,\rho} 
        \leq \frac{1}{c - \norm{K}{1,\rho_0,\mathrm{unif}}}\norm{F}{2,\rho}.
    \]
    Finally, to show causality, assume that $F$ vanishes on some interval $]-\infty,a]$ for some $a\in \R$.
    Since $C_\rho$ and $T_K$ are causal, we infer that $Q$ maps functions vanishing on $]-\infty,a]$ to functions vanishing on the same set.
    Since $U = \lim_{n \to \infty} Q^n(0)$, we infer $U=0$ on $]-\infty,a]$ as well.
  \end{proof}

  \begin{rem}
    For kernels $K_\ell(t,s) := \ell(t-s)$ with $\ell \in L^{1,\mu}(\R_{\geq 0})$ for some $\mu \in \R$,
    the assumption $\norm{K_\ell}{1,\rho_0,\mathrm{unif}} < c$ is always satisfied if we choose $\rho_0$ large enough.
    Indeed, since
    \[
      \norm{K_\ell}{1,\rho,\mathrm{unif}}
        = \norm{\ell}{1,\rho} 
        = \int_0^\infty |\ell(t)| e^{-\rho t} \, \dt,
    \]
    the claim follows by the monotone convergence theorem.

  \end{rem}

  \begin{prob}
   So far, the authors could neither prove nor disprove that $\|K\|_{1,\rho,\mathrm{unif}} \to 0$ for $\rho \to \infty$ holds for general
    kernels $K \in L^{1,\rho_0}_{\mathrm{unif}}(\Lambda)$. We leave this as an open question.
  \end{prob}

  \section{Numerical method}\label{sec:method}
  The numerical method used in this work is a combination of a discontinuous Galerkin finite element method in time
  and a conforming Galerkin finite element method in space. Both types of methods have a long history. 
  Discontinuous Galerkin methods were first introduced in 1973 \cite{RH73} for neutron transport and later extended 
  to more general problems, see, e.g., \cite{Riviere08}. The Galerkin finite element method is even older, 
  going back to a paper by Galerkin~\cite{Galerkin1915} in 1915. There is a very broad literature on this topic,
  some very good text books are \cite{BreSco02, Ci02, EG21a}. 
  
  The method is the same as in \cite{FrTW16} and we will start the description of the time discretisation.   
  We define time points $0=t_0<t_1<\dots<t_M=T$ and corresponding time intervals $I_m=(t_{m-1},t_m]$
  of width $\tau_m=t_m-t_{m-1}$. Let $q\in \N$ be a given polynomial degree in time. Then we define the space 
  \[
    \U{H}\coloneqq 
      \{u\in L^{2,\rho}(\R;H)\,:\,\forall m\in \{1,\ldots,M\}:
        u|_{I_m}\in \PS_q(I_m;H)\},
  \]
  where we denote by 
  \[
    \PS_q(I_m;H)\coloneqq 
      \lin \{ I_m \ni t\mapsto t^k \zeta \in H; k\in \{0,\ldots,q\}, \zeta\in H\}
  \]
  the space of $H$-valued polynomials of degree at most $q$ defined on $I_m$.   
  With the scalar product 
  \[
    \scprm{p,q} \coloneqq 
      \intop_{t_{m-1}}^{t_m} 
      \scp{p(t),q(t)} \exp(-2\rho(t-t_{m-1})) \dt
  \]
  the space $\PS_q(I_m;H)$ becomes a Hilbert space. 
  
  The numerical method itself will
  use a discretised version of that scalar product by applying a quadrature rule instead of the integral.
  We denote by ${\omega}_i^m$ and $\hat{t}^m_i$, $i\in\{0,\dots,q\}$,
  the weights and nodes of a weighted Gau\ss--Radau formula with $q+1$ nodes
  on the reference time interval $\widehat{I}=(-1,1]$, such that
  \[
    \int_{\widehat I} \e^{-2\rho \tau_m (t+1)}p(t)\dt 
      = \sum_{i=0}^q {\omega}^m_ip(\hat{t}^m_i)
  \]
  holds for all polynomials $p$ of degree at most $2q$ (see e.g. \cite{TW2016}).
  
  A standard linear transformation $T_m: \widehat{I}\to I_m$ and
  transformed Gau\ss--Radau points $\tmi := T_m(\hat{t}^m_i)$, 
  $i\in\{0,\dots,q\}$ then give a quadrature formula on $I_m$, defined by
  \[
    \Qmr{v} := \sum_{i=0}^q {\hat\omega}^m_i v(\tmi),
  \]
  where ${\hat \omega}^m_i:=\frac{\tau_m}{2} {\omega}^m_i$.
  It follows 
  $
    \Qmr{p} = \scprm{p,1}
  $
  for all polynomials of degree at most $2q$.
  
  Additionally, we use the same points $\tmi$ in the definition of an interpolation operator 
  into $\U{H}$. Let $\varphi_{m,i}$ with $i\in\{0,\dots,q\}$ be the associated polynomial Lagrange 
  basis functions on $I_m$ to the nodes $\tmi$. Then for a function $v\in C([0,T],H)$ we define by
  \begin{equation}\label{eq:P}
    (Pv)(0) = v(0),\quad
    (Pv)\big|_{I_m}(t)
    = \sum_{i=0}^q v(\tmi)\varphi_{m,i}(t), \qquad m\in\{1,\dots,M\},
  \end{equation}
  an interpolation operator $P$ in time. Note that by definition it holds
  \[
    \Qmr{v}=\Qmr{Pv}.
  \]

  We arrive at a first discrete formulation of \eqref{eq:problem} by considering each interval $I_m$ separately
  and using 
  \[
    \Qmr{a,b}:=\Qmr{\scp{a,b}},
  \]
  instead of the scalar products $\scprm{a,b}$:
  
  For given $F\in \U{H}$ and $x_0\in H$, find $U^\tau\in\U{H}$, such that for 
  all $\Phi\in \U{H}$ and $m\in\{1,2,\dots,M\}$ it holds
  \begin{equation}\label{eq:discr_quad_form}
    \Qmr{(\partial_t M_0+M_1+A+T_K)U^\tau,\Phi}
      +\scp{M_0 \jump{U^\tau}_{m-1}^{x_0},{\Phi}^+_{m-1}}
      =\Qmr{ F,\Phi },
  \end{equation}
  where 
  \[
    \jump{U^\tau}_{m-1}^{x_0}\coloneqq
    \begin{cases}
      U^\tau(t_{m-1}^+)-U^\tau(t_{m-1}-),& m\in\{2,\ldots,M\}\\ 
      U^\tau(t_0^+)-x_0,& m=1,
    \end{cases} 
  \]
  denotes the jump at $t_{m-1}$ and $\Phi^+_{m-1}\coloneqq \Phi(t_{m-1}^+)$ the right-hand side limit of $\Phi$ at $t_{m-1}$.

  In our analysis we will use the discretised norms
  \[
    \norm{v}{Q,\rho,m}^2\coloneqq\Qmr{v,v}
    \quad\text{and}\quad
    \norm{v}{Q,\rho}^2\coloneqq\sum_{m=1}^M\Qmr{v,v}\e^{-2\rho t_{m-1}}
  \]
  as approximations of $\norm{v}{\rho,m}^2\coloneqq \intop_{I_m} |v(t)|_H^2 
  \exp(-2\rho(t-t_{m-1}))\dt$ and $\norm{v}{\rho}^2$. Note that for $v\in\U{H}$ 
  the approximation is exact.

  \begin{rem}\label{rem:Reformulation}
    The original problem can also be posed in an unweighted $L_2$-space instead of the weighted $L_2$-space. 
    Indeed, let $V(t)=\e^{-\rho t}U(t)$ and rewrite the original problem. This leads to the weak formulation
    \[
      \scp[m]{(\pt_t M_0+(\rho M_0+M_1)+A+\widetilde T_K)V,W}=\scp[m]{\widetilde F,W}
    \]
    where $\widetilde F(t)=\e^{\rho t}F(t)$ and $\widetilde T_K V=\int_0^t \e^{\rho(t-s)}K(t,s)V(s)\ds=:\int_0^t\widetilde K(t,s)V(s)\ds$.
    Now the scalar products can be discretised by a standard right-sided Gauß--Radau quadrature rule
    and we obtain \eqref{eq:discr_quad_form} without the weights, but with changed operators that depend on $\rho$.
  \end{rem}
  
  Note that in \eqref{eq:discr_quad_form} the quadrature applied to the first three operators depends only 
  on $U^\tau_m = U^\tau|_{(t_{m-1},t_m]}$, 
  while the one applied to $T_KU^\tau$ is non-local as $T_KU^\tau|_{(t_{m-1},t_m]}$ depends on $U^\tau|_{(0,t_m]}$.
  Therefore, let us rewrite the $T_KU^\tau$ term for $\tmi\in(t_{m-1},t_m]$:
  \begin{align*}
    (T_KU^\tau)(\tmi)
      &= \int_0^{\tmi} K(\tmi,s)U^\tau(s)\ds
       =:\sum_{n=1}^m J_nU^\tau_n(\tmi)
  \end{align*}
  where
  \[
    J_nU^\tau_n(\tmi):=\int_{t_{n-1}}^{\min\{\tmi,t_n\}}K(\tmi,s)U^\tau_n(s)\ds.
  \]
  Then the quadrature formulation for computing $U^\tau_m$ is given by:
  \begin{multline}
    B_m(U^\tau_m,\Phi):=
    \Qmr{(\partial_t M_0+M_1+A)U^\tau_m,\Phi}
    +\Qmr{J_m U^\tau_m,\Phi}\\
      +\scp{M_0 \jump{U^\tau}_{m-1}^{x_0},{\Phi}^+_{m-1}}
      =\Qmr{ F,\Phi }-\sum_{n=1}^{m-1}\Qmr{J_n U^\tau_n,\Phi}.\label{eq:semidiscr_quad_form_final}
  \end{multline}
  
  In practice, we also need to approximate the integral operator. We will use an adaptive quadrature method, 
  that provides results up to machine precision, thus making them almost exact, even for weakly singular kernels, as in the case of fractional integrals. 
  Alternatively, one could use another fixed quadrature rule for this discretisation and analyse its influence.
  
  The method provided so far is only a semi-discretisation in time. For the full discretisation the choice of discrete 
  spaces for the spatial variable depends on the given operator $A$, since we use conforming discrete spaces -- 
  that is to say they are discrete subspaces of $\dom(A)$.
  
  We discretise $\Omega\subseteq \R^d$ into a set $\Omega_h$
  of pairwise disjoint simplices of maximum diameter $h$ covering $\Omega$. For $d=1$ this is similar to the
  partitioning in time.
  In many applications the unbounded operator $A$ is of the form $A=\spmtrx{0&C^*\\-C&0}$, where $C$ is an operator of vector calculus,
  see \cite{Picard,McGPTW20,STW22}. Let us denote its associated solution component by $U_C\in\dom(C)$. 
  
  Below, we provide some examples of discrete subspace of $\dom(C)$ for certain operators $C$,
  along with the associated interpolation operators required for the subsequent numerical analysis.
  
  Suppose $C$ is
  \begin{itemize}
    \item the \textbf{gradient}, then $\dom(C)=H^1(\Omega)$ and the classical choice are continuous, piecewise polynomial Lagrange-elements
          $\PS_k(\Omega_h)$ of maximum degree $k$. As interpolation operator $I$ we can choose a classical pointwise defined interpolation operator 
          or a quasi-interpolator of Scott--Zhang type \cite{SZ90}. For both it holds
          \[
            \norm{U_C-IU_C}{L^2(\Omega)}+\norm{\Grad(U_C-IU_C)}{L^2(\Omega)}
              \leq c h^k,
          \]
          assuming $U_C\in H^{k+1}(\Omega)$.
    \item the \textbf{divergence}, then $\dom(C)=H(\Div,\Omega)$; that is, the space of all $L^2$-vector-fields 
          such that its distributional divergence is represented by an $L^2$-function, and we can use 
          Raviart--Thomas-elements $RT_{k-1}(\Omega_h)=(\PS_{k-1}(\Omega_h))^d+\vx\PS_{k-1}(\Omega_h)$.
          For the standard interpolation operator $I$ into $RT_{k-1}(\Omega_h)$ it holds
          \[
            \norm{U_C-IU_C}{L^2(\Omega)}+\norm{\Div(U_C-IU_C)}{L^2(\Omega)}
              \leq c h^{k},
          \]
          assuming $U_C\in (H^{k}(\Omega))^d$ such that $\Div U_C\in H^{k}(\Omega)$, see e.g.~\cite{EG21a}.
    \item the \textbf{curl}, then $\dom(C)=H(\Curl,\Omega)$ (similarly, this is the space of all $L^2$-vector-fields 
          such that its distributional rotation is represented by an $L^2$-vector-field) and we can use N\'{e}d\'{e}lec elements (of the first kind)
          $N_{k-1}(\Omega_h)$, defined as
          \[
            N_{k-1}(\Omega_h)=(\PS_{k-1}(\Omega_h))^d\oplus S_{k}(\Omega_h),\quad
            S_{k}(\Omega_h)=\{\vq\in(\PS_{k}^H(\Omega_h))^d\colon \vq\cdot\vx=0\},
          \]
          where $\PS_k^H(\Omega_h)$ are the homogeneous polynomials of degree $k$. For the standard interpolation operator $I$ into $N_{k-1}(\Omega_h)$
          it holds
          \[
            \norm{U_C-IU_C}{L^2(\Omega)}+\norm{\Curl(U_C-IU_C)}{L^2(\Omega)}
              \leq c h^k,
          \]
          assuming $U_C\in (H^{k}(\Omega))^d$ such that $\Curl U_C\in (H^{k}(\Omega))^d$, see e.g.~\cite{EG21a}.
  \end{itemize}
  We did not include boundary conditions in above examples, but they can also be implemented in 
  these discrete spaces in a canonical way.
  Choosing for all components of $U$ the corresponding conforming finite element space,
  we have a composite space $V_k(\Omega)$. Denoting the composite interpolation operator into $V_k(\Omega)$ by $I$ too,
  we have
  \begin{gather}
    \norm{U-IU}{H}+\norm{A(U-IU)}{H}
      \leq c h^k,\label{eq:interspace}
  \end{gather}
  if $U\in H^k(\Omega)$ and $AU\in H^k(\Omega)$, where we omitted the dimension of the spaces. 
  
  Our fully discrete space can now be defined by
  \[
    \U{V_k}
      :=\left\{
          U\in H_\rho(\R,H):
          U|_{I_m}\in\PS_q(I_m,V_k(\Omega)),
          m\in\{1,\dots,M\}
       \right\},
  \]
  which does not have to be continuous in time. Now \eqref{eq:semidiscr_quad_form_final} together with the space $\U{V_k}$ instead of $\U{H}$
  gives the final numerical method for this paper: Find $U^{\tau,h}\in\U{V_k}$, such that for all $1\leq m\leq M$ and $\Phi\in \U{V_k}$ it holds
  \begin{multline}
    B_m(U^{\tau,h}_m,\Phi):=
    \Qmr{(\partial_t M_0+M_1+A)U^{\tau,h}_m,\Phi}
    +\Qmr{J_m U^{\tau,h}_m,\Phi}\\
      +\scp{M_0 \jump{U^{\tau,h}}_{m-1}^{x_0},{\Phi}^+_{m-1}}
      =\Qmr{ F,\Phi }-\sum_{n=1}^{m-1}\Qmr{J_n U^{\tau,h}_n,\Phi}.\label{eq:discr_quad_form_final}
  \end{multline}

  \section{Numerical Analysis}\label{sec:numanalysis}
  We follow the standard procedure for estimating errors by first separating them into interpolation 
  errors and a discrete errors, using the triangle inequality:
  \[
    \norm{U^{\tau,h}-U}{\rho}\leq \norm{U^{\tau,h}-PIU}{\rho}+\norm{U-PIU}{\rho}=:\norm{\xi}{\rho}+\norm{\eta}{\rho},
  \]
  where $\xi\coloneqq U^{\tau,h}-PIU$ and $\eta\coloneqq U-PIU$. Note that $U^{\tau,h}-U=\xi-\eta$ and
  $\xi\in\U{V_k}$. We will derive an error equation for this function, that couples
  it, via the bilinear form, to the interpolation error $\eta$. The final step is then to estimate the 
  interpolation error.
  
  We also provide some details on how the method is implemented.
  
  \subsection{Error equation and coercivity}
  The discrete formulation \eqref{eq:discr_quad_form_final} is not only solved by the discrete function $U^{\tau,h}$, but also 
  by the exact solution $U$ of \eqref{eq:problem}:
  \[
    \overline{(\pt_t M_0+M_1+A+T_K)}U=F.
  \]
  So, by subtracting \eqref{eq:discr_quad_form_final} for $U^{\tau,h}$ and $U$ for the test function $\xi$,
  and splitting the result into discrete and interpolation error,
  we obtain the \textbf{error equation}
  \begin{multline}\label{eq:error}
    E_\textnormal{d}^m
    :=\Qmr{(\partial_t M_0+M_1)\xi,\xi}
     +\scp{M_0 \jump{\xi}_{m-1}^0,\xi^+_{m-1}}+\Qmr{T_K\xi,\xi}\\
     =\Qmr{(\partial_t M_0+M_1+A)\eta,\xi}
     +\scp{M_0 \jump{\eta}_{m-1}^0,\xi^+_{m-1}}+\Qmr{T_K\eta,\xi}
     =:E_\textnormal{i}^m
  \end{multline}
  for all $m\in\{1,\ldots,M\}$.
  Note that we used $\Qmr{A\xi,\xi}=0$, because $A$ is skew-selfadjoint, to simplify the equation.
  The analysis of \cite[Lemma 3.1]{FrTW16} 
  is directly applicable and yields
  \begin{multline*}
    \Qmr{(\partial_t M_0+M_1)\xi,\xi}+\scp{M_0 \jump{\xi}_{m-1}^0,\xi^+_{m-1}} \\
    \geq \gamma\norm{\xi}{\rho,m}^2+
         \frac{1}{2}
              \left[
                 \scp{M_0\xi^-_m,\xi^-_m}\e^{-2\rho\tau_m}
                -\scp{M_0\xi^-_{m-1},\xi^-_{m-1}}
                +\scp{M_0\jump{\xi}_{m-1}^0,\jump{\xi}_{m-1}^0}
              \right],    
  \end{multline*}
  where $\xi^-_m\coloneqq \xi(t_{m}^-)$ and $\xi^-_0\coloneqq 0$. To arrive at the full 
  discrete norm $\norm{\xi}{\rho}$ in the lower bound, we need a weighted summation over $m$. This gives
  \begin{align}\label{eq:erroreq:1}
    \sum_{m=1}^M \e^{-2\rho t_{m-1}}E_\textnormal{d}^m   &\geq \sum_{m=1}^M \e^{-2\rho t_{m-1}}(\gamma \|\xi\|_{\rho,m}^2+\Qmr{T_K\xi,\xi})+\\
    &\quad +\frac{1}{2} \sum_{m=1}^M \left[
                 \scp{M_0\xi^-_m,\xi^-_m}\e^{-2\rho t_m}
                -\scp{M_0\xi^-_{m-1},\xi^-_{m-1}}\e^{-2\rho t_{m-1}}
                \right]\\ \nonumber
      &= \gamma\norm{\xi}{\rho}^2 +\frac{1}{2}\scp{M_0\xi_M^-,\xi_M^-}+ \sum_{m=1}^M \e^{-2\rho t_{m-1}}\Qmr{T_K\xi,\xi}\\
      &\geq \gamma\norm{\xi}{\rho}^2 + \sum_{m=1}^M \e^{-2\rho t_{m-1}}\Qmr{T_K\xi,\xi},
  \end{align}
  where we neglected positive contributions from the jumps  and at $t_M=T$ in the first and last estimate,
  and used $\norm{\xi}{Q,\rho}=\norm{\xi}{\rho}$ for $\xi\in\U{V_k}$.
  
  We are left with estimating the integral term and showing it is basically a perturbation. Recall the norm
  \[
    \norm{K}{1,\rho,\mathrm{unif}}=\max\left\{\esssup_{t\in[0,T]}\int_0^t |K(t,s)|\e^{-\rho(t-s)}\ds,\,\esssup_{s\in[0,T]}\int_s^T |K(t,s)|\e^{-\rho(t-s)}\dt\right\},
  \]
  that was used in establishing well-posedness in Section \ref{sec:wellposedness}. In our numerical method we have replaced the
  scalar products by a quadrature rule. If we now want to estimate the integral term, 
  we also need to apply a quadrature rule in the definition of the above norm. To be more precise, we need
  this only for the second term in the norm and we show the derivation of its discrete counterpart.
  We start by rewriting it:
  \begin{align*}
    \int_s^T |K(t,s)|\e^{-\rho(t-s)}\dt &= \sum_{m=1}^M \int_{t_{m-1}}^{t_m} \id{[s,t_m]}(t)|K(t,s)|\e^{-\rho(t-s)}\dt\\
           &= \sum_{m=1}^M \e^{-2\rho t_{m-1}}\int_{t_{m-1}}^{t_m} \id{[s,t_m]}(t)|K(t,s)|\e^{\rho(t+s)}\e^{-2\rho(t-t_{m-1})}\dt.
  \end{align*}
  Now the sum has the structure needed for the approximation by the quadrature $\Qmr{\cdot}$. We obtain our new semi-discrete norm
  \begin{align*}
    \norm{K}{Q,1,\rho,\mathrm{unif}}
      :=\max\bigg\{&\esssup_{t\in[0,T]}\int_0^t |K(t,s)|\e^{-\rho(t-s)}\ds,\\
                   &\esssup_{s\in[0,T]}\sum_{m=1}^M \e^{-2\rho t_{m-1}}\Qmr{\id{[s,t_m]}|K(\cdot,s)|\e^{\rho(\cdot+s)}}\bigg\}
  \end{align*}
  and we can estimate the operator in this semi-discrete norm.
  \begin{lem}\label{lem:JQbound}
    Let $V\in L^2_\rho([0,T];H)$. We have
    \[
      \norm{T_KV}{Q,\rho} \leq \norm{K}{Q,1,\rho,\mathrm{unif}}\norm{V}{\rho}.
    \]
  \end{lem}
  \begin{proof}
    By the definition 
    of the quadrature rule,
    and using ideas from the analysis in the proof of Proposition \ref{prop:Opbound} we have
    \begin{align*}
        \Qmr{T_KV,T_KV}
        &= \sum_{i=0}^q {\hat\omega}_i^m\left| \int_0^\tmi K(\tmi,s)V(s)\ds \right|^2\\
        &= \sum_{i=0}^q {\hat\omega}_i^m\left| \int_0^\tmi K(\tmi,s)\e^{-\frac{\rho}{2}(\tmi-s)}V(s)\e^{\frac{\rho}{2} (\tmi-s)}\ds \right|^2\\
        &\leq \sum_{i=0}^q {\hat\omega}_i^m \int_0^\tmi |K(\tmi,s)|\e^{-\rho(\tmi-s)}\ds
                                                      \int_0^\tmi |K(\tmi,s)|\e^{\rho(\tmi-s)}|V(s)|^2\ds\\
        &\leq \norm{K}{Q,1,\rho,\mathrm{unif}}\sum_{i=0}^q {\hat\omega}_i^m \int_0^\tmi |K(\tmi,s)|\e^{\rho(\tmi+s)}|V(s)|^2\e^{-2\rho s}\ds \\
        &= \norm{K}{Q,1,\rho,\mathrm{unif}}\int_0^T \sum_{i=0}^q {\hat\omega}_i^m \id{[0,\tmi]}(s) |K(\tmi,s)|\e^{\rho(\tmi+s)}|V(s)|^2\e^{-2\rho s}\ds \\
        &= \norm{K}{Q,1,\rho,\mathrm{unif}}\int_0^T \sum_{i=0}^q {\hat\omega}_i^m \id{[s,t_m]}(\tmi) |K(\tmi,s)|\e^{\rho(\tmi+s)}|V(s)|^2\e^{-2\rho s}\ds\\
        &= \norm{K}{Q,1,\rho,\mathrm{unif}}\int_0^T \Qmr{\id{[s,t_m]} |K(\cdot,s)|\e^{\rho(\cdot+s)}}|V(s)|^2\e^{-2\rho s}\ds.
    \end{align*}
    Applying the weighted sum we obtain
    \begin{align*}
      \norm{T_KV}{Q,\rho}^2 &= \sum_{m=1}^M \e^{-2\rho t_{m-1}}\Qmr{T_KV,T_KV}\\
        &\leq \norm{K}{Q,1,\rho,\mathrm{unif}}\int_0^T \sum_{m=1}^M \e^{-2\rho t_{m-1}}\Qmr{\id{[s,t_m]} |K(\cdot,s)|\e^{\rho(\cdot+s)}}|V(s)|^2\e^{-2\rho s}\ds\\
        &\leq \norm{K}{Q,1,\rho,\mathrm{unif}}^2\int_0^T |V(s)|^2\e^{-2\rho s}\ds\\
        &= \norm{K}{Q,1,\rho,\mathrm{unif}}^2\norm{V}{\rho}^2.\qedhere
    \end{align*}
  \end{proof}  
  
  Now we are ready to estimate the remaining term in \eqref{eq:erroreq:1}. 
  \begin{lem}\label{lem:Jxi_new}
    Let $V\in\U{V_k}$. We have
    \[
      \left|\sum_{m=1}^M \e^{-2\rho t_{m-1}}\Qmr{T_KV,V}\right| \leq \norm{K}{Q,1,\rho,\mathrm{unif}}\norm{V}{\rho}^2.
    \]
  \end{lem}
  \begin{proof}
    As a first step we apply the Cauchy--Schwarz inequality
    \begin{align*}
      |\Qmr{T_KV,V}|^2
        &\leq \Qmr{T_KV,T_KV}\Qmr{V,V}
         = \Qmr{T_KV,T_KV}\norm{V}{Q,\rho,m}^2\\
        &= \Qmr{T_KV,T_KV}\norm{V}{\rho,m}^2
    \end{align*}
    and the definition of the norms for $V\in\U{V_k}$.
    Now a further Cauchy--Schwarz inequality gives
    \begin{align}
      \left|\sum_{m=1}^M \e^{-2\rho t_{m-1}}\Qmr{T_KV,V}\right|^2
        &\leq \left(\sum_{m=1}^M \e^{-2\rho t_{m-1}}\Qmr{T_KV,T_KV}\right)\norm{V}{\rho}^2
         =  \norm{T_KV}{Q,\rho}^2\norm{V}{\rho}^2.
    \end{align}
    The  desired estimate follows by the previous Lemma~\ref{lem:JQbound}.
  \end{proof}  
  
  Combining \eqref{eq:erroreq:1} and Lemma~\ref{lem:Jxi_new} we have the coercivity result
  \begin{gather}\label{eq:erroreq:2}
    \sum_{m=1}^M \e^{-2\rho t_{m-1}}E_\textnormal{d}^m
      \geq \frac{\gamma}{2}\norm{\xi}{\rho}^2
  \end{gather}
  if $\norm{K}{Q,1,\rho,\mathrm{unif}}\leq\frac{\gamma}{2}$.
  
  \subsection{Convergence analysis}
  Before estimating the right-hand side of \eqref{eq:error}, we introduce 
  a shorthand notation for the discretised full norm in space and time
  \[
    \norm{u}{Q,\rho,k,D}^2
    = \sum_{m=1}^M\Qm{|u|^2_{k,D}}\e^{-2\rho t_{m-1}}
  \]
  where $D\subseteq \Omega$ is measurable.

  We rewrite the derivative term in \eqref{eq:error} as shown in \cite[Lemma 3.7]{FrTW16} as
  \[
    \Qmr{\pt_t M_0\eta,\xi}+\scp{M_0\jump{\eta}_{m-1},\xi_{m-1}^+}
      = \Qmr{\pt_t M_0(U-\widehat P IU),\xi}+R(U,\xi),
  \]
  where the temporal Lagrange interpolation operator $\widehat P$ utilises $t_{m-1}$ 
  in addition to $\tmi$, $i\in\{0,\dots,q\}$ as interpolation point, 
  see \cite{FrTW16} for the precise definition. The remainder $R(U,\xi)$ has the bound
  \[
    |R(U,\xi)|
      \leq C\alpha\tau_m|M_0\eta_{m-1}^+|_H^2
            +\beta\norm{\xi}{\rho,m}^2,
  \]
  where here and throughout $\alpha,\beta>0$ such that $\alpha\beta=\frac{1}{4}$,
  and the precise value of $\beta$ is given below. Now an application of the Cauchy--Schwarz inequality
  followed by the Young inequality to all the quadrature terms in $E_\textnormal{i}^m$ yields
  \begin{align*}
    E_\textnormal{i}^m 
      & \leq \norm{\pt_t M_0(U-\widehat P IU)}{Q,\rho,m}\norm{\xi}{Q,\rho,m}+|R(U,\xi)|
             +\norm{M_1\eta}{Q,\rho,m}\norm{\xi}{Q,\rho,m}\\&\qquad
             +\norm{A\eta}{Q,\rho,m}\norm{\xi}{Q,\rho,m}
             +\norm{T_K\eta}{Q,\rho,m}\norm{\xi}{Q,\rho,m}\\
      &\leq \alpha\left(\norm{\pt_t M_0(U-\widehat P IU)}{Q,\rho,m}^2
                         +C\tau_m|M_0\eta_{m-1}^+|_H^2
                         +\norm{M_1\eta}{Q,\rho,m}^2
                         +\norm{A\eta}{Q,\rho,m}^2
                         +\norm{T_K\eta}{Q,\rho,m}^2\right)\\&\qquad+5\beta\norm{\xi}{\rho,m}.
  \end{align*}
  In view of ~\eqref{eq:erroreq:2}, we set $\beta=\frac{\gamma}{20}$ and therefore $\alpha=\frac{5}{\gamma}$. 
  Now summation, the error equation~\eqref{eq:error} and 
  inequality \eqref{eq:erroreq:2} give the relation between discrete error and interpolation errors
  \begin{multline}
    \frac{\gamma}{4}\norm{\xi}{\rho}^2
       \leq C\bigg(\norm{\pt_t M_0(U-\widehat P IU)}{Q,\rho}^2
                         +T\max_{1\leq m\leq M}\{|M_0\eta_{m-1}^+|_H^2\e^{-2\rho t_{m-1}}\}\\
                         +\norm{M_1\eta}{Q,\rho}^2
                         +\norm{A\eta}{Q,\rho}^2
                         +\norm{T_K\eta}{Q,\rho}^2  \bigg).\label{eq:estimate}
  \end{multline}
  \begin{rem}
    Like in \cite{FrTW16} we directly have control over the end-point error $|\xi(T)|$ by the above right-hand side
    using \eqref{eq:erroreq:1} with the final time contribution included. The control of the supremum in time,
    like in \cite{FrTW16}, is not considered here, but seems plausible and is observed numerically.
  \end{rem}
  The first four terms on the right-hand side of \eqref{eq:estimate} were estimated in \cite{FrTW16}. 
  We recall the results here, assuming sufficient regularity of $U$ and setting $\tau=\max\limits_{1\leq m\leq M}\tau_m$.
  \begin{itemize}
    \item A generalisation of \cite[Lemma 4.2]{FrTW16} (given for the special case of 
          $A=\spmtrx{0&\Div\\\Grad_0&0}$ where $\Grad_0$ is the gradient for $H_0^1(\Omega)$-functions 
          and $\Div$ is its negative adjoint operator), using \eqref{eq:interspace}, provides
          \[
            \norm{A\eta}{Q,\rho} 
              \leq C h^k \norm{AU}{Q,\rho,k,\Omega}.
          \]
    \item \cite[Lemma 4.3]{FrTW16} gives
          \[
            \norm{M_1\eta}{Q,\rho} 
              \leq C h^k \norm{M_1U}{Q,\rho,k,\Omega}.
          \]
    \item \cite[Lemma 4.5]{FrTW16} gives
          \[
            \norm{\pt_t M_0(U-\widehat P IU)}{Q,\rho} 
              \leq C \left( h^k \norm{M_0\pt U}{Q,\rho,k,\Omega}+\tau^{q+1}\sup_{t\in[0,T]}|M_0\pt_t^{q+2}IU|_H\right).
          \]
    \item \cite[Lemma 4.6]{FrTW16} gives
          \[
            \max_{1\leq m\leq M}\{|M_0\eta_{m-1}^+|_H\e^{-\rho t_{m-1}}\}
              \leq C \tau^{q+1}\sup_{t\in[0,T]}|M_0\pt_t^{q+1}IU|_H.
          \]
    \item In addition \cite[Lemma 4.2]{FrTW16} also gives
          \[
            \norm{\eta}{Q,\rho} 
              \leq C h^k \norm{U}{Q,\rho,k,\Omega}.
          \]
  \end{itemize}
  The only term left to estimate is the integral term. Here we can use Lemma~\ref{lem:JQbound} for $V=\eta$ and 
  obtain
  \[
    \norm{T_K\eta}{Q,\rho} \leq \norm{K}{Q,1,\rho,\mathrm{unif}}\norm{\eta}{\rho}.
  \]
  This inequality uses the non-discretised norm, that can be
  estimated by
  \[
    \norm{\eta}{\rho}
      \leq \norm{U-PU}{\rho}+\norm{PU-PIU}{\rho}
      = \norm{U-PU}{\rho}+\norm{U-IU}{Q,\rho}
      = \norm{U-PU}{\rho}+\norm{\eta}{Q,\rho}.
  \]
  The second term is estimated above, while for the first we use \cite[Theorem 2.5]{Fr23}:
  \[
    \norm{U-PU}{\rho}\leq C\tau^{q+1}\norm{\pt_t^{q+1}U}{L^2_{\rho}(0,T;H)}.
  \]
  Combining above estimates, we obtain the main result of this paper.
  \begin{thm}\label{thm:main}
    We assume for the solution $U$ of \eqref{eq:problem} the 
    regularity     
    \[
      U\in H_\rho^{1}(\R;H^k(\Omega))\cap 
           H_\rho^{q+3}(\R;L^2(\Omega)) 
    \]
    as well as 
    \[
      AU\in H_\rho(\R; H^k(\Omega)).
    \]
    If $K$ satisfies
    \[
      \norm{K}{Q,1,\rho,\mathrm{unif}}\leq \frac{\gamma}{2},
    \]
    then we have for the discrete error of the numerical solution given by 
    \eqref{eq:discr_quad_form_final}
    \begin{align*}
      \norm{\xi}{\rho}
        \leq C(h^k + T^{1/2}\tau^{q+1})
    \end{align*}
    and therefore for the error itself
    \begin{align*}
      \norm{U-U^{\tau,h}}{\rho}
        \leq C(h^k + T^{1/2}\tau^{q+1}).
    \end{align*}
  \end{thm}
  
  \subsection{On the implementation}
  In this section we will share some of the implementation details used when programming 
  the numerical method. As a basis, we used the finite-element framework $\mathbb{SOFE}$ in Matlab.
  
  For a numerical simulation all integrals must be approximated. Here, the scalar products
  with test functions in space and time are discretised using a weighted, right-sided Gauß--Radau 
  quadrature rule in time with $q+1$ evaluations, where $q$ is the polynomial degree in time. 
  The quadrature weights and nodes are non-standard due to the exponentially 
  weighted $L^2$-space, but can always be computed numerically as shown in \cite[Chapter 4.6]{PTVF07}.
  An alternative is the reformulation given in Remark~\ref{rem:Reformulation}, 
  where only standard quadrature rules are needed. Their weights and nodes are tabulated
  for small degrees $q$ or can be computed as referred to above.

  A Gauß--Legendre quadrature rule with $k+1$ evaluations is used to approximate the integration in space, where $k$ is the polynomial degree in space.
  The third type of integral can be found in the integral operator $T_K$. We implemented two possibilities for its evaluation.
  For smooth kernels, we use a Gauß--Legendre quadrature rule with a fixed number of $q+1$ evaluations 
  to evaluate the integral with high precision.
  However, for kernels with a (weak) singularity this fixed order is insufficient to produce satisfactory results.
  In this case, we apply the Matlab routine \texttt{integral}, which employs an adaptive quadrature rule and 
  automatically refines the integration interval near singularities. The downside to this more precise approximation
  is higher computational costs. This is particularly evident in the computation of the history terms.
  
  However, evaluating the history terms is always the most time-consuming part of the computational step. 
  The number of evaluations grows quadratically in the number of total time steps $M$. Alternatives to this direct 
  approach exist, for instance in the numerical simulation of Caputo fractional derivatives, 
  where, in addition to the integral operator, a derivative of the unknown function is also included. 
  These fast evaluation methods (see e.g.~\cite{JZZZ17}) transform power-functions
  into exponential series and use specialised quadrature rules to reuse previously computed data in a recursion.
  This reduces the computational costs from $\ord{M^2}$ to $\ord{M\log(M)}$.  
  It would be interesting to apply these methods to our integral operators.
  
  The Gauß--Radau quadrature rule used in time is exact for polynomials of degree $2q$.
  In our weak formulation (see \eqref{eq:discr_quad_form_final}) the terms involving the matrices $M_0$ and $M_1$  
  only include polynomials of degree $2q-1$ and $2q$, resp., if we assume that $M_0$ and $M_1$ are constant. 
  Thus, these scalar products are evaluated exactly. However, the integral operator term is different.
  Here, the polynomial basis is integrated first (assuming a constant kernel $K$ for ease of presentation), thus increasing
  the polynomial degree by one. When multiplied by a test function, the integrand is of degree $2q+1$, 
  for which the quadrature rule is no longer exact. Interestingly, using a better quadrature rule
  for this integral gives essentially the same results -- the difference is below the approximation quality. 

  \section{Numerical examples}\label{sec:examples}
  Our theoretical results suggest convergence of order $q+1$ in time and $k$ in space. To 
  match these, we set $q=k-1$. For the meshes in time and space, we chose equidistant meshes with $M$ cells in time
  and $N$ cells in each spatial dimension. This gives a time step of $\tau=\frac{T}{M}$ and a mesh width in space of
  $h=\frac{1}{N}$ for $\Omega=(0,1)^d$.
  
  We begin our investigation of the numerical behaviour with a case where the exact solution is known in advance.
  We compute the right-hand sides and solve the resulting system numerically. 
  The advantage of a given solution is that errors can be computed exactly.
  In the following, we denote by $\partial_x$ the weak derivative operator on $L^2(0,1)$ with maximal domain
  $H^1(0,1)$, while $\partial_x^\circ$ denotes the restriction of $\partial_x$ to the domain $H_0^1(0,1)$. 
  Note that then $\partial_x^\ast=-\partial_x^\circ$.
  
  \textbf{Example 1:} Let us consider the 1d problem
  \[
    \left(\pt_t \pmtrx{1&0\\0&1}+\pmtrx{1&0\\0&1}+\pmtrx{[c]0&\pt_x\\\pt_x^\circ&0}+T_K\right)\pmtrx{u\\v}=F
  \]
  in $[0,T]\times\Omega=[0,2]\times(0,1)$,
  where the integral operator is given by 
  \[
    T_K\pmtrx{u\\v}(t,x) = \int_0^tK(t,s)\pmtrx{u\\v}(s,x)\ds
  \]
  with the smooth and bounded kernel
  \[
    K(t,s) = \pmtrx{[c]t-s & s\\t & (t-s)^2}
  \]
  and the right hand side $F$, such that
  \begin{align*}
    u(t,x) &= (t+\e^{-t})\sin(\pi x^2),\\
    v(t,x) &= \cos(t)\exp(x)
  \end{align*}
  is the exact solution. 
  For this example, we have $\gamma=1$ and
  \[
    \norm{K}{\rho,1,\mathrm{unif}}\leq \frac{T}{\rho} = \frac{2}{\rho} \stackrel{!}{\leq}\frac{\gamma}{2}=\frac{1}{2}
    \rarrow \rho\geq 4.
  \]
  For the semidiscrete norm the same bound holds assuming the mesh being fine enough. While the integrals
  in the above norm can be evaluated analytically, the quadrature sum is more difficult. But we can compute
  the norm numerically and look into the behaviour w.r.t.~ $\rho$. Doing so, we observe the same bound.

  From a practical point of view the precise value of $\rho$ does not matter and we 
  observe numerically that any value of $\rho\geq0$ produces very similar results. 
  Thus, we chose $\rho=1$.
  
  As the given kernel is smooth, we use a Gauß--Legendre quadrature rule with $q+1$ points
  in order to evaluate the history terms faster. Furthermore, in one spatial dimension we can use 
  $\PS_{k}^{\textrm{cont}}(\Omega_h)$-elements.
  Table~\ref{tab:ex1}
  \begin{table}[tb]
    \caption{Convergence results for $U-U_h$ of Example 1\label{tab:ex1}}
    \begin{center}
      {\begin{tabular}{crllll}
        \toprule
        $(k,q)$ & $N=M$ & 
        \multicolumn{2}{c}{$E_{\sup}(U-U_h)$} &
        \multicolumn{2}{c}{$\norm{U-U_h}{\rho}$}\\
        \midrule
        $(1,0)$
        &    8 & 3.848e-01 &      & 9.477e-02 &     \\
        &   16 & 1.940e-01 & 0.99 & 4.698e-02 & 1.01\\
        &   32 & 9.776e-02 & 0.99 & 2.351e-02 & 1.00\\
        &   64 & 4.912e-02 & 0.99 & 1.177e-02 & 1.00\\
        &  128 & 2.463e-02 & 1.00 & 5.894e-03 & 1.00\\
        \midrule
        $(2,1)$
        &    8 & 1.635e-02 &      & 5.182e-03 &     \\
        &   16 & 4.490e-03 & 1.86 & 1.311e-03 & 1.98\\
        &   32 & 1.175e-03 & 1.93 & 3.310e-04 & 1.99\\
        &   64 & 3.006e-04 & 1.97 & 8.335e-05 & 1.99\\
        &  128 & 7.603e-05 & 1.98 & 2.092e-05 & 1.99\\
        \midrule
        $(3,2)$
        &    8 & 9.613e-04 &      & 2.588e-04 &     \\
        &   16 & 7.126e-05 & 3.75 & 1.557e-05 & 4.05\\
        &   32 & 7.520e-06 & 3.24 & 1.189e-06 & 3.71\\
        &   64 & 9.082e-07 & 3.05 & 1.160e-07 & 3.36\\
        &  128 & 1.132e-07 & 3.00 & 1.331e-08 & 3.12\\
        \bottomrule
      \end{tabular}}
    \end{center}
  \end{table}
  shows the convergence results for polynomial degrees $q\in\{0,1,2\}$ and $k=q+1$.
  Here
  \[
      E_{\sup}(v)^2 \coloneqq \sup_{t\in[0,T]}\scp{M_0v(t),v(t)}
  \]
  denotes the supremum norm that we did not estimate here, but in \cite{FrTW16}.
  Furthermore, the norm $\norm{U-U_h}\rho$ 
  is estimated using a refined quadrature rule in the final columns.
  Note that the expected convergence rates of $\min\{k,q+1\}$ are achieved.

  \textbf{Example 2:} We use the same data as in Example 1, except for the kernel, where we take
  \[
    K(t,s) = \pmtrx{[c](t-s)^{-3/4} & 0\\0 & (t-s)^{-1/2}}.
  \]
  This kernel is not bounded, but we have
  \[
    \norm{K}{Q,\rho,1,\mathrm{unif}}\leq C \rho^{-1/4},
  \]
  so for large enough $\rho$ the numerical convergence theory does
  still apply. Practically, $\rho=1$ suffices and
  we obtain the results given in Table~\ref{tab:ex2_new}.
  \begin{table}[tb]
    \caption{Convergence results for $U-U_h$ of Example 2\label{tab:ex2_new}}
    \begin{center}
      {\begin{tabular}{crllll}
        \toprule
        $(k,q)$ & $N=M$ & 
        \multicolumn{2}{c}{$E_{\sup}(U-U_h)$} &
        \multicolumn{2}{c}{$\norm{U-U_h}{\rho}$}\\
        \midrule
        $(1,0)$
          &   8 & 3.242e-01 &      & 7.924e-02 &     \\
          &  16 & 1.776e-01 & 0.87 & 3.933e-02 & 1.01\\
          &  32 & 9.395e-02 & 0.92 & 1.974e-02 & 0.99\\
          &  64 & 4.859e-02 & 0.95 & 9.908e-03 & 0.99\\
          & 128 & 2.479e-02 & 0.97 & 4.964e-03 & 1.00\\
        \midrule
        $(2,1)$
          &   8 & 1.612e-02 &      & 4.577e-03 &     \\
          &  16 & 4.464e-03 & 1.85 & 1.172e-03 & 1.97\\
          &  32 & 1.173e-03 & 1.93 & 2.970e-04 & 1.98\\
          &  64 & 3.004e-04 & 1.97 & 7.484e-05 & 1.99\\
          & 128 & 7.600e-05 & 1.98 & 1.878e-05 & 1.99\\
        \midrule
        $(3,2)$
          &   8 & 6.029e-04 &      & 1.840e-04 &     \\
          &  16 & 5.936e-05 & 3.34 & 1.182e-05 & 3.96\\
          &  32 & 7.169e-06 & 3.05 & 1.017e-06 & 3.54\\
          &  64 & 8.982e-07 & 3.00 & 1.095e-07 & 3.22\\
          & 128 & 1.129e-07 & 2.99 & 1.309e-08 & 3.06\\ 
         \bottomrule
      \end{tabular}}
    \end{center}
  \end{table}
  Although the kernel is weakly singular, wo observe the theoretical convergence rates of $\min\{k,q+1\}$.
  Due to the weak singularity of the kernel we do not employ a Gauß--Legendre quadrature rule for 
  the history terms. We would need a very high order method to give equally precise results compared to 
  the adaptive quadrature routine of Matlab.

  \textbf{Example 3:} For our third and final example we prescribe a given right-hand side and an initial 
  condition. We do not know the exact solution and use a 
  reference solution for the computation of errors. 
  This is an approximation, computed on a fine mesh with 
  twice the number of cells in time and space compared to the finest mesh used in the 
  computation run, with polynomial degrees $q_{ref}=q+1$ and $k_{ref}=k+1$.
  Provided the true solution is sufficiently smooth, we can expect the reference solution 
  to be a useful substitute.
  
  The problem is the same as in Example 2, but with the right-hand side
  $F$ given as $F=1$ and the initial condition
  \[
    u(0,x)=\sin(2\pi x^2),\,v(0,x)=0. 
  \]
  Thus, we use the same kernel
  \[
    K(t,s) = \pmtrx{[c](t-s)^{-3/4} & 0\\0 & (t-s)^{-1/2}},
  \]
  but the exact solution is unknown. An approximation of the solution to this problem is shown in Figure~\ref{fig:ex3_sol}.
  \begin{figure}[tb]
    \begin{center}
       \includegraphics[width=0.4\textwidth]{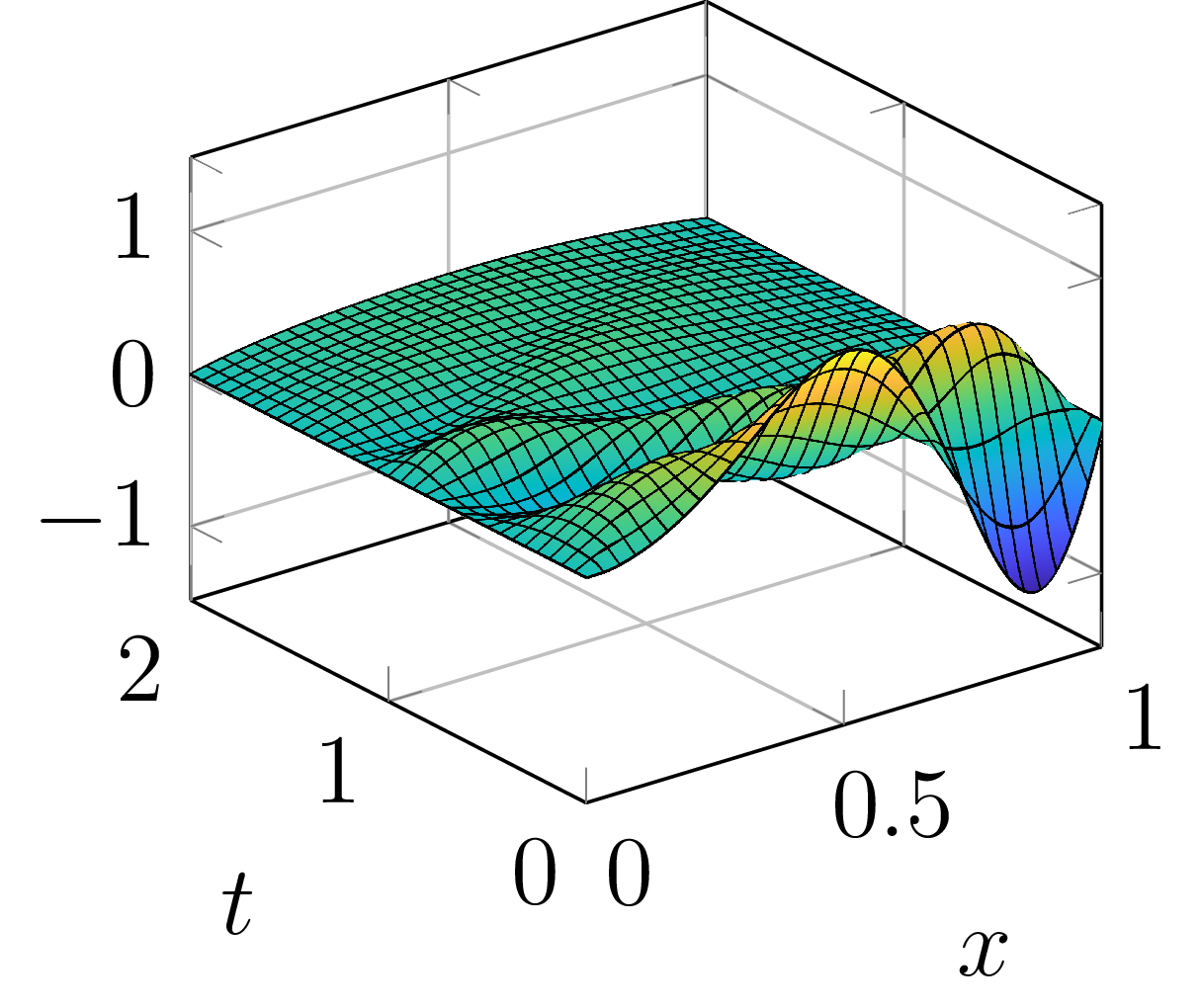}
       \includegraphics[width=0.4\textwidth]{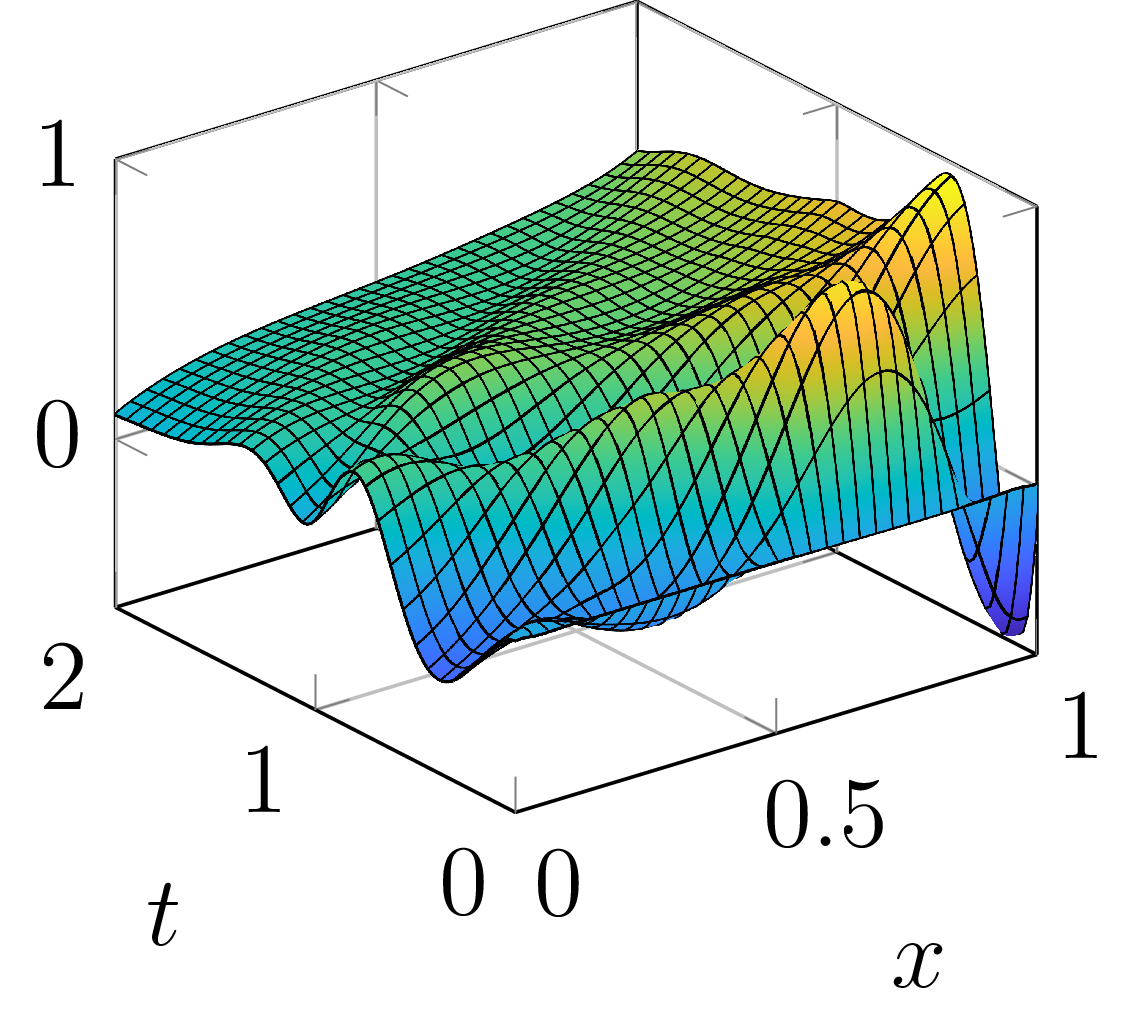}
    \end{center}
    \caption{Solution $u$ (left) and $v$ (right) of Example 3\label{fig:ex3_sol}}
  \end{figure}
  Using the reference solution, we obtain the convergence results shown in Table~\ref{tab:ex3_new}.
  \begin{table}[tb]
    \caption{Convergence results for $U-U_h$ of Example 3\label{tab:ex3_new}}
    \begin{center}
      {\begin{tabular}{crllll}
        \toprule
        $(k,q)$ & $N=M$ & 
        \multicolumn{2}{c}{$E_{\sup}(U-U_h)$} &
        \multicolumn{2}{c}{$\norm{U-U_h}{\rho}$}\\
        \midrule
        $(1,0)$
         &   8 & 4.912e-01 &      & 1.838e-01 &     \\
         &  16 & 3.887e-01 & 0.34 & 1.391e-01 & 0.40\\
         &  32 & 2.629e-01 & 0.56 & 9.786e-02 & 0.51\\
         &  64 & 1.527e-01 & 0.78 & 6.310e-02 & 0.63\\
         & 128 & 8.104e-02 & 0.91 & 3.759e-02 & 0.75\\
        \midrule
        $(2,1)$
         &   8 & 2.333e-01 &      & 7.772e-02 &     \\
         &  16 & 9.896e-02 & 1.24 & 2.498e-02 & 1.64\\
         &  32 & 3.197e-02 & 1.63 & 6.377e-03 & 1.97\\
         &  64 & 9.372e-03 & 1.77 & 1.647e-03 & 1.95\\
         & 128 & 2.709e-03 & 1.79 & 4.799e-04 & 1.78\\
        \midrule
        $(3,2)$
         &   8 & 8.059e-02 &      & 1.982e-02 &     \\
         &  16 & 1.833e-02 & 2.14 & 3.190e-03 & 2.64\\
         &  32 & 3.405e-03 & 2.43 & 7.128e-04 & 2.16\\
         &  64 & 7.663e-04 & 2.15 & 2.611e-04 & 1.45\\
         & 128 & 2.360e-04 & 1.70 & 1.059e-04 & 1.30\\
         \bottomrule
      \end{tabular}}
    \end{center}
  \end{table}
  They clearly demonstrate that we do not achieve the theoretical order of convergence. 
  While the convergence rate seems to approach the optimal order for the lower polynomial degrees,
  this is not true for the higher order case.
  One possible cause of this lack of convergence order may be a lack of regularity in the solution. 
  In this case, the interpolation error estimates are not applicable, and the concept of 
  a reference solution as a substitute for the real solution is imperfect.

\newpage

  \section{Conclusion}
  The existence theory of Rainer Picard for first order systems $(\partial_t M_0+M_1+A)U=F$
  can be extended to problems involving integral operators $T_K$. We considered here the easier case of
  an integral operator without time derivative. The inclusion of $\partial_t T_K$
  is subject of further research. Also, numerical methods obtained for the linear case
  can be extended similarly to this new set of problems and the same convergence orders
  can be proved.
  \bibliographystyle{abbrvurl}
  \bibliography{lit}

@Article{FrTW16,
  title                    = {Numerical methods for changing type systems},
  author                   = {Franz, S. and Trostorff, S. and Waurick, M.},
  journal                  = {IMAJNA},
  Year                     = {2019},
  Number                   = {2},
  Pages                    = {1009--1038},
  Volume                   = {39},
  doi                      = {doi.org/10.1093/imanum/dry007}
}

@article{Fr23,
  Title                    = {Post-processing and improved error estimates of numerical methods for evolutionary systems},
  Author                   = {Franz, S},
  Journal                  = {IMA J. Numer. Anal.},
  Year                     = {2024},
  Volume                   = {44}, 
  Number                   = {5},
  Pages                    = {2936--2958},
  doi                      = {10.1093/imanum/drad082}
}

@article{JZZZ17,
  title={Fast Evaluation of the Caputo Fractional Derivative and its Applications to Fractional Diffusion Equations},
  author={Jiang, S. and Zhang, J. and Zhang, Q. and Zhang, Z.},
  journal={Communications in Computational Physics},
  volume={21},
  number={3},
  pages={650--678},
  year={2017},
  doi={10.4208/cicp.OA-2016-0028}
}

@Book{PTVF07,
  Title                    = {Numerical Recipes 3rd Edition: The Art of Scientific Computing},
  Author                   = {Press, W. H. and Teukolsky, S. A. and Vetterling, W. T. and Flannery, B. P.},
  Publisher                = {Cambridge University Press},
  Year                     = {2007},

  Address                  = {New York, NY, USA},
  Edition                  = {3},
  doi                      = {10.1145/1874391.187410},
}

@Unpublished{RH73,
  Title                    = {{Triangular mesh methods for the neutron transport equation}},
  Author                   = {W. H. Reed and T. R. Hill},
  Note                     = {Submitted to American Nuclear Society Topical Meeting on Mathematical Models and Computational Techniques for Analysis of Nuclear Systems, Los Alamos Laboratory},
  Year                     = {1973}
}

@Book{Riviere08,
  Title                    = {Discontinuous Galerkin Methods for Solving Elliptic and Parabolic Equations},
  Author                   = {Rivière, B.},
  Publisher                = {Society for Industrial and Applied Mathematics},
  Year                     = {2008},
  doi                      = {10.1137/1.9780898717440},
}

@Article{Galerkin1915,
  author  = {B.G. Galerkin},
  title   = {Series solution of some problems of elastic equilibrium of rods and plates (in Russian)},
  journal = {Vestn. Inzh. Tech.},
  year    = {1915},
  volume  = {19},
  pages   = {897--908},
}

@Book{BreSco02,
  author    = {Brenner, S. C. and Scott, L. R.},
  title     = {The mathematical theory of finite element methods},
  edition   = {Second},
  publisher = {Springer-Verlag},
  series    = {Texts in Applied Mathematics},
  volume    = {15},
  address   = {New York},
  year      = {2002},
  doi       = {10.1007/978-1-4757-3658-8},
}

@Book{Ci02,
  author    = {P.G. Ciarlet},
  title     = {The Finite Element Method for Elliptic Problems},
  publisher = {SIAM},
  address   = {Philadelphia},
  year      = {2002},
  doi       = {10.1137/1.9780898719208},
}

@book{EG21a,
    AUTHOR = {Ern, Alexandre and Guermond, Jean-Luc},
     TITLE = {Finite elements {I} -- {A}pproximation and interpolation},
    SERIES = {Texts in Applied Mathematics},
    VOLUME = {72},
 PUBLISHER = {Springer, Cham},
      YEAR = {2021},
     PAGES = {xii+325},
      ISBN = {978-3-030-56340-0; 978-3-030-56341-7},
   MRCLASS = {65-01},
  MRNUMBER = {4242224},
       DOI = {10.1007/978-3-030-56341-7},
       URL = {https://doi.org/10.1007/978-3-030-56341-7},
}

@Article{SZ90,
  Title    = {{Finite element interpolation of nonsmooth functions satisfying boundary conditions}},
  author   = {Scott, L. R. and Zhang, S.},
  Journal  = {Math. Comp.},
  Year     = {1990},
  Number   = {54},
  Pages    = {483--493},
  doi      = {10.1090/S0025-5718-1990-1011446-7}
}

@Article{Picard,
  Title    = {{A structural observation for linear material laws in classical mathematical physics}},
  Author   = {R. Picard},
  Journal  = {Math. Methods Appl. Sci.},
  Year     = {2009},
  Number   = {14},
  Pages    = {1768--1803},
  Volume   = {32},

  Language = {English},
  doi      = {10.1002/mma.1110}
}

@Book{McGPTW20,
  author    = {McGhee, D. AND Picard, R. AND Trostorff, S. AND Waurick, M.},
  publisher = {Springer International Publishing},
  title     = {A Primer for a Secret Shortcut to PDEs of Mathematical Physics},
  year      = {2020},
  edition   = {1st},
  booktitle = {Frontiers in Mathematics},
  doi       = {10.1007/978-3-030-47333-4},
}

@Book{STW22,
author    = {Seifert, Chr. and Trostorff, S. and Waurick, M.},
title     = {Evolutionary Equations: Picard's Theorem for Partial Differential Equations, and Applications},
year      = {2022},
publisher = {Springer International Publishing},
isbn      = {978-3-030-89397-2},
doi       = {10.1007/978-3-030-89397-2_10},
}

@Article{Trostorff2015,
  author  = {Sascha Trostorff},
  title   = {On Integro-Differential Inclusions with Operator-valued Kernels.},
  journal = {Math. Meth. Appl. Sci.},
  year    = {2015},
  volume  = {38},
  number  = {5},
  pages   = {834-850},
  doi     = {10.1002/mma.3111},
}

@PhdThesis{Trostorff_Habil,
  author = {Sascha Trostorff},
  title  = {Exponential Stability and Initial Value Problems for Evolutionary Equations.},
  school = {TU Dresden, Germany},
  year   = {2018},
  type   = {Habilitation Thesis},
  note   = {\url{https://nbn-resolving.org/urn:nbn:de:bsz:14-qucosa-236494}},
}

@Article{Trostorff2013,
  author  = {Sascha Trostorff},
  title   = {Autonomous Evolutionary Inclusions with Applications to Problems with Nonlinear Boundary Conditions.},
  journal = {Int. J. Pure Appl. Math.},
  year    = {2013},
  volume  = {85},
  number  = {2},
  pages   = {303-338},
  doi     = {10.12732/ijpam.v85i2.10},
}

@Article{PTWW2013,
  author  = {Rainer Picard and Sascha Trostorff and Marcus Waurick and Maria Wehowski},
  title   = {On Non-autonomous Evolutionary Problems.},
  journal = {J. Evol. Equ.},
  year    = {2013},
  volume  = {13},
  pages   = {751-776},
  doi     = {10.1007/s00028-013-0201-7}
}

@Article{Trostorff2020,
  author  = {Sascha Trostorff},
  title   = {Well-posedness for a general class of differential inclusions.},
  journal = {J. Differ. Equations},
  year    = {2020},
  volume  = {268},
  number  = {11},
  pages   = {6489-6516},
  doi     = {10.1016/j.jde.2019.11.045}
}

@Article{PTW2015,
  author  = {Rainer Picard and Sascha Trostorff and Marcus Waurick},
  title   = {On Evolutionary Equations with Material Laws Containing Fractional Integrals.},
  journal = {Math. Meth. Appl. Sci.},
  year    = {2015},
  volume  = {38},
  number  = {15},
  pages   = {3141-3154},
  doi     = {10.1002/mma.3286}
}

@TechReport{TW2016,
  author      = {Sascha Trostorff and Marcus Waurick},
  title       = {On the weighted Gauß-Radau Quadrature},
  institution = {TU Dresden},
  year        = {2016},
  note        = {arXiv:1610.09016},
}

@article{Lubich2004,
  author    = {Lubich, Christian},
  title     = {Convolution Quadrature Revisited},
  journal   = {BIT Numerical Mathematics},
  volume    = {44},
  number    = {3},
  pages     = {503--514},
  year      = {2004},
  doi       = {10.1023/B:BITN.0000046813.23911.2d},
}

@article{Schaedle2006,
  author    = {Sch\"adle, Achim and L{\'o}pez-Fern{\'a}ndez, Mar{\'i}a and Lubich, Christian},
  title     = {Fast and Oblivious Convolution Quadrature},
  journal   = {SIAM Journal on Scientific Computing},
  volume    = {28},
  number    = {2},
  pages     = {421--438},
  year      = {2006},
  doi       = {10.1137/050623139},
}

@article{ZhangVandewalle2006,
  author    = {Zhang, Chun and Vandewalle, Stefan},
  title     = {General Linear Methods for Volterra Integro-differential Equations with Memory},
  journal   = {SIAM Journal on Scientific Computing},
  volume    = {27},
  number    = {6},
  pages     = {2010--2031},
  year      = {2006},
  doi       = {10.1137/040607058},
}

@book{Brunner2004,
  author    = {Brunner, Hermann},
  title     = {Collocation Methods for Volterra Integral and Related Functional Differential Equations},
  publisher = {Cambridge University Press},
  address   = {Cambridge},
  year      = {2004},
  doi       = {10.1017/CBO9780511543234},
}

@article{Sheng2014,
  author    = {Sheng, Chao-Tsung and Wang, Zhong and Guo, Ben-Yu},
  title     = {A Multistep Legendre–Gauss Spectral Collocation Method for Nonlinear Volterra Integral Equations},
  journal   = {SIAM Journal on Numerical Analysis},
  volume    = {52},
  number    = {4},
  pages     = {1953--1980},
  year      = {2014},
  doi       = {10.1137/130915200},
}

@book {Diet10,
    AUTHOR = {Diethelm, Kai},
     TITLE = {The analysis of fractional differential equations},
    SERIES = {Lecture Notes in Mathematics},
    VOLUME = {2004},
 PUBLISHER = {Springer-Verlag, Berlin},
      YEAR = {2010},
     PAGES = {viii+247},
      ISBN = {978-3-642-14573-5},
       DOI = {10.1007/978-3-642-14574-2},
}

@article {KopMC19,
    AUTHOR = {Kopteva, Natalia},
     TITLE = {Error analysis of the {L}1 method on graded and uniform meshes
              for a fractional-derivative problem in two and three
              dimensions},
   JOURNAL = {Math. Comp.},
  FJOURNAL = {Mathematics of Computation},
    VOLUME = {88},
      YEAR = {2019},
    NUMBER = {319},
     PAGES = {2135--2155},
      ISSN = {0025-5718},
       DOI = {10.1090/mcom/3410},
}
\end{document}